\documentclass[11pt]{article}
\usepackage{amsmath,amsthm,amsfonts,latexsym,amsopn,verbatim,amscd,amssymb}
\theoremstyle{plain}
\newtheorem{thm}{Theorem}[section]
\newtheorem{lemma}[thm]{Lemma}

\newtheorem{proposition}[thm]{Proposition}

\theoremstyle{definition}

\newcommand{\Z}{\mathbb{Z}}

\newcommand{\bnum}{\begin{enumerate}}
\newcommand{\enum}{\end{enumerate}}

\numberwithin{equation}{section}

\DeclareMathOperator{\Cent}{Cent}
\begin{document}
\title{\textbf{A note on $n$-centralizer finite rings}}
\author{Jutirekha Dutta, Dhiren K. Basnet and Rajat K. Nath\footnote{Corresponding author}}
\date{}
\maketitle
\begin{center}\small{\it 
Department of Mathematical Sciences, Tezpur University,\\ Napaam-784028, Sonitpur, Assam, India.\\

Emails:\, jutirekhadutta@yahoo.com,  dbasnet@tezu.ernet.in and rajatkantinath@yahoo.com*}
\end{center}

\medskip

\begin{abstract} 
Let $R$ be a finite ring and let $\Cent(R)$ denote the set of all  distinct centralizers of $R$.  $R$ is called an $n$-centralizer ring if $|Cent(R)| = n$. In this paper, we characterize $n$-centralizer finite rings for $n \leq 7$. 
\end{abstract}

\medskip

 

\medskip

\section{Introduction}
Let $F$ be an algebraic system having finite number of elements which is closed under a multiplication operation. Let $\Cent(F)$ denote the set $\{C(x) : x \in F\}$ where $C(x) = \{y \in F : xy = yx\}$ is called the centralizer of $x$ in $F$. $F$ is called $n$-centralizer if $|\Cent(F)| = n$. The study of finite $n$-centralizer group was initiated by Belcastro and Sherman \cite{bG94} in the year 1994. Since then many mathematician have studied $n$-centralizer group for different values of $n$. Characterizations of finite $n$-centralizer groups for $n = 4, 5, 6, 7, 8$ and $9$ can be found in \cite{ASA2007, ashrafi00, ashrafi2000,  ashrafi2006, baishya, bG94, Dutta13}. In \cite{Jdutta}, the authors have initiated the study of finite $n$-centralizer rings and characterized $n$-centralizer finite rings for $n = 4, 5$. Further, Dutta has extended the characterizations of $4$-centralizer groups and rings  to infinite case  in \cite{Dutta13} and \cite{Dutta14} respectively. In this paper, we give a characterization of  $n$-centralizer finite rings for $n = 6, 7$ including a new characterization for $n = 4, 5$. 
The problem of characterizing finite rings 
has received considerable attention in recent years (see \cite{bBhK14, Cu10, dOP94, Fin93, gC95} etc.) starting from the works of  Eldridge \cite{El68} and   Raghavendran \cite{Ra70}.


\section{Preliminaries}
In this section first we describe some notations and prove some  results which are useful for the subsequent sections. Throughout this paper $R$ denotes a finite ring. For any subring $S$ of $R$,   $R/S$ or $\frac{R}{S}$  denotes the additive quotient group  and $|R : S|$ denotes the index of  the additive subgroup $S$ in the additive group $R$. Note that the isomorphisms considered are the additive group isomorphisms. Also for any two non-empty subsets $A$ and $B$ of  $R$, we define $A + B := \{a + b : a \in A, b \in B\}$. We shall use the fact that  for any non-commutative ring $R$, the additive group $\frac{R}{Z(R)}$ is not a cyclic group (see \cite[Lemma 1]{dmachale}) where $Z(R) = \{ s \in R \;:\; rs = sr \text{ for all } r \in R\}$ is the center of $R$. 
  
 In \cite{Jdutta}, we have observed that there is no $2, 3$-centralizer ring. We also have the  following result.

\begin{thm}\label{4cent}
Let $R$ be a non-commutative ring. Then
\begin{enumerate}
\item $|\Cent(R)| = 4$ if and only if $|R : Z(R)| = 4$ \emph{(see \cite[Theorem 3.2]{Jdutta})}.
\item $|\Cent(R)| = 5$ if and only if $|R : Z(R)| = 9$ \emph{(see \cite[Theorem 4.3]{Jdutta})}.
\end{enumerate}
\end{thm}

 
A family of proper subgroups  of a group $G$ is called a cover of $G$ if $G$ is the union of those subgroups. A cover $X = \{H_1, H_2, \dots, H_k\}$ of $G$ is called irredundant if no proper subset of $X$ is also a cover of $G$. In this paper, we shall use frequently the following two results regarding cover and irredundant cover of a group, proved by Tomkinson \cite{tom87}.

\begin{lemma}[\cite{tom87}, Lemma $3.3$]\label{tomlemma}
Let $M$ be a proper subgroup of the (finite) group $G$ and let $H_1, H_2, \dots, H_k$ be subgroups with $|G : H_i| = \beta_i$ and $\beta_1 \leq \beta_2 \leq \dots \leq \beta_k$. If $G = M \cup H_1 \cup H_2 \cup \dots \cup H_k$, then $\beta_1 \leq k$.

Furthermore, if $\beta_1 = k$, then $\beta_1 = \beta_2 = \dots = \beta_k = k$ and $H_i \cap H_j \leq M$, for all $i \neq j$.
\end{lemma}

\begin{thm}[\cite{tom87}, Theorem $3.4$]\label{tomtheorem}
Let $\{X_1, X_2, \dots, X_n\}$ be an irredundant cover of  a group $G$   such that $|G : X_i| = \alpha_i$ where $\alpha_1 \leq \alpha_2 \leq \dots \leq \alpha_k$ and $D = \overset{n}{\underset{i = 1}{\cap}}X_i$. Then
\begin{enumerate}
\item $\alpha_2 \leq n - 1$.
\item If $\alpha_2 = n - 1$ then $|G : D| \leq (n - 1)^2(n - 3)!$.
\item If $\alpha_2 < n - 1$ then $|G : D| \leq (n - 2)^3(n - 3)!$.
\end{enumerate}
\end{thm}

	Let $F$ be an algebraic system having finite number of elements which is closed under a multiplication operation.
A nonempty subset $T = \{r_1, r_2, \dots$, $r_t\}$ of $F$ is called a \textit{set of pairwise non-commuting elements} if $r_ir_j \neq r_jr_i$ for all   $1 \leq i \neq j \leq t$. A set of pairwise non-commuting elements of $F$ is said to be of \textit{maximal size} if its cardinality is the largest one among all such sets.

This extends Definition 2.1 of \cite{ASA2007}. Considering $F$ to be a finite group Abdollahi et al. \cite{ASA2007} have obtained several results regarding set of pairwise non-commuting elements  having maximal size (see Remark 2.1, Lemma 2.4, Proposition 2.5 and Lemma 2.6 of \cite{ASA2007}). In this section, we consider $F$ to be a finite ring and obtain the following analogous results.

\begin{proposition}\label{intersection Z(R)}
Let $\{r_1, r_2, \dots, r_t\}$ be a set of pairwise non-commuting elements of $R$ having maximal size. Then 
\bnum
\item $R = C_R(r_1) \cup C_R(r_2) \cup \dots \cup C_R(r_t)$ 
\item $\overset{t}{\underset{i = 1}{\cap}}C_R(r_i) = Z(R)$.
\item $\{C_R(r_i) : i = 1, 2, \dots, t\}$ is an irredundant cover  of $R$.
\enum
\end{proposition}
\begin{proof}
(a) Suppose there exists an element $r \in R - (\overset{t}{\underset{i = 1}{\cup}}C_R(r_i))$, then $\{r_1, r_2$, $\dots, r_t, r\}$ is a set of $(t + 1)$-pairwise non-commuting elements of $R$, which is a contradiction. Hence $R = \overset{t}{\underset{i = 1}{\cup}}C_R(r_i)$.

(b)  Suppose there exists an element $r \in (\overset{t}{\underset{i = 1}{\cap}}C_R(r_i)) - Z(R)$, then $rs \neq sr$ for some $s \in R$ . Now for each $i = 1, 2, \dots, t$, we define $y_i = r_i$ if $r_is \neq sr_i$ and $y_i = r + r_i$ if $r_is = sr_i$. Then $\{y_1, y_2, \dots, y_t, s\}$ is a set of $(t + 1)$-pairwise non-commuting elements of $R$, which is a contradiction. Hence, $\overset{t}{\underset{i = 1}{\cap}}C_R(r_i) = Z(R)$.

(c) If not, then there exists a proper non-commutative centralizer $C_R(r_i)$ for some $i, 1 \leq i \leq t$ such that $C_R(r_i) \subseteq \overset{t}{\underset{j = 1, i \neq j}{\cup}} C_R(r_j)$. This contradicts the hypothesis that $\{r_1, r_2, \dots, r_t\}$ is a set of pairwise non-commuting elements of $R$.   
\end{proof}

\begin{proposition}\label{thm1}
Let $\{r_1, r_2, \dots, r_t\}$ be a set of pairwise non-commuting elements of a ring $R$ having maximal size. Then 
\bnum
\item $t \geq 3$ 
\item $t + 1 \leq |\Cent(R)|$
\item If $R$ is a ring such that every proper centralizer is commutative then for all $a, b \in R - Z(R)$ either $C_R(a) = C_R(b)$ or $C_R(a) \cap C_R(b) = Z(R)$.
\enum
\end{proposition}
\begin{proof}
(a) Given $R$ is not commutative, so there exist $r, s \in R - Z(R)$ such that $rs \neq sr$. Thus $\{r, s, r + s\}$ is a set of pairwise non-commuting elements of $R$ and so $t \geq 3$.

(b) Given $r_1, r_2, \dots, r_t$ is a set of pairwise non-commuting elements of $R$, so $C_R(r_1), C_R(r_2), \dots, C_R(r_t)$ all are distinct. Hence, $t + 1 \leq |\Cent(R)|$.

(c) Suppose for some $a, b \in R - Z(R)$ we have $C_R(a) \cap C_R(b) \neq Z(R)$. Then there exists an element $c \in C_R(a) \cap C_R(b) - Z(R)$. Then $a, b \in C_R(c)$ and so for any $r \in C_R(a)$ we have $r \in C_R(c)$, since $C_R(a)$ is commutative. Therefore $C_R(a) \subseteq C_R(c)$. Similarly, it can be seen that $C_R(c) \subseteq C_R(a), C_R(b) \subseteq C_R(c)$ and $C_R(c) \subseteq C_R(b)$. Hence $C_R(c) = C_R(a) = C_R(b)$. 
\end{proof}

\begin{proposition}\label{thm3}
Let $R$ be an $n$-centralizer ring such that the cardinality of pairwise non-commuting elements of $R$ is $n - 1$. Then for every non-central elements $r$ and $s$ of $R$, $C_R(r) = C_R(s)$ or $C_R(r) \cap C_R(s) = Z(R)$. 
\end{proposition}
\begin{proof}
Let $\{r_1, r_2, \dots, r_{n - 1}\}$ be a set of pairwise non-commuting elements of $R$. Then $C_R(r_1), C_R(r_2), \dots, C_R(r_{n - 1})$ are all distinct and so by Proposition \ref{thm1}(b), $\Cent(R) = \{R, C_R(r_1), C_R(r_2), \dots, C_R(r_{n - 1})\}$. Suppose for some $i, 1 \leq i \leq n - 1, C_R(r_i)$ is non-commutative. So there exist elements $a, b \in C_R(r_i)$ such that $ab \neq ba$. So $C_R(a) \neq R, C_R(b), C_R(r_i)$. Without any loss, we can assume that $C_R(a) = C_R(r_j)$ for some $j \neq i, 1 \leq j \leq (n - 1)$. This gives $r_i \in C_R(a) = C_R(r_j)$ and so $r_ir_j = r_jr_i$, a contradiction. Hence every proper centralizer is commutative. So by Proposition \ref{thm1}(c), we have the required result. 
\end{proof}





We conclude this section by a new characterization of $4,5$-centralizer finite rings which is analogous to \cite[Lemma 2.4]{ASA2007}. 
The following lemma is useful in proving Theorem \ref{thm2}.
\begin{lemma}\label{C(x)commutative}
Let $R$ be a ring and $|R : Z(R)| = 4$ or $8$. Then $C_R(x)$ is commutative for all $x \in R - Z(R)$.
\end{lemma}
\begin{proof}
 For any $x \in R - Z(R)$ we know that, $Z(R) \subseteq Z(C_R(x)) \subseteq C_R(x) \subset R$. Therefore
\[
\frac{C_R(x)/Z(R)}{Z(C_R(x))/Z(R)} \cong \frac{C_R(x)}{Z(C_R(x))}
\]
and so $|\frac{C_R(x)}{Z(C_R(x))}| = 2$ or $1$. Hence $C_R(x)$ is commutative.
\end{proof}

\begin{thm}\label{thm2}
Let $\{r_1, r_2, \dots, r_t\}$ be a set of pairwise non-commuting elements of a ring $R$ having maximal size. Then 
\bnum
\item $|\Cent(R)| = 4$ if and only if $t = 3$, 
\item $|\Cent(R)| = 5$ if and only if $t = 4$.
\enum
\end{thm}
\begin{proof}
(a) Suppose $|\Cent(R)| = 4$. By  Proposition \ref{thm1}, we have that $t \geq 3$ and $t + 1 \leq |\Cent(R)|$ which gives $t \leq 3$. Hence $t = 3$. 

Conversely, suppose that $t = 3$. Then by Proposition \ref{intersection Z(R)} and Theorem \ref{tomtheorem}, we have 
\[
|R : Z(R)| \leq  \max \{(3 - 1)^2(3 - 3)!, (3 - 2)^3(3 - 3)!\} = 4
\]
 which gives $|R : Z(R)| = 4$. Hence, by  Theorem \ref{4cent}(a), it follows that $|\Cent(R)| = 4$.

(b) Suppose $|\Cent(R)| = 5$. By  Proposition \ref{thm1}, we have that $t \geq 3$ and $t + 1 \leq |\Cent(R)|$ which gives $t = 3$ or $4$. If $t = 3$ then by part (a), we have $|\Cent(R)| = 4$, a contradiction. Hence $t = 4$.
Conversely, suppose that $t = 4$. Then by Proposition \ref{intersection Z(R)} and Theorem \ref{tomtheorem}, we have 
\[
|R : Z(R)| \leq  \max \{(4 - 1)^2(4 - 3)!, (4 - 2)^3(4 - 3)!\} = 9
\]
 which gives $|R : Z(R)| = 4, 8$ or $9$. Suppose $|R : Z(R)| = 4$ or $8$ then by Lemma \ref{C(x)commutative} we have  $C_R(r)$ is commutative for any $r \in R - Z(R)$. So, by Proposition \ref{thm1} we have  $C_R(r) \cap C_R(s) = Z(R)$ for any $r, s \in R - Z(R)$ and $r \neq s$. Also we have that for any $r \in R - Z(R)$, $|R : C_R(r)|$ divides $|R : Z(R)|$. So if $\beta_2 = 3$ then $3$ divides $|R : Z(R)|$, a contradiction. Hence $\beta_2 = 2 = \beta_1$. Therefore $|R| \geq |C_R(r_1) + C_R(r_2)|$ and so $|R| \geq \frac{(|R|/2)(|R|/2)}{|Z(R)|}$ which gives $|R : Z(R)| = 4$. Hence by Theorem \ref{4cent} (a), we have $|\Cent(R)| = 4$ and so by part(a), $t = 3$, a contradiction. Therefore $|R : Z(R)| = 9$, so by Theorem \ref{4cent} (b) we have $|\Cent(R)| = 5$.
This completes the theorem.
\end{proof}

\section{$6$-centralizer finite rings}
In this section, we have the following  characterization of $6$-centralizer finite rings.
\begin{thm}\label{6cent}
Let $R$ be a $6$-centralizer finite ring. Then $|R : Z(R)| = 8, 12$ or $16$.
\end{thm}
\begin{proof}
Given $|\Cent(R)| = 6$. Let $\{r_1, r_2, \dots, r_t\}$ be a set of pairwise non-commuting elements of $R$ having maximal size. Then $C_R(r_i), 1 \leq i \leq t$ are distinct proper centralizers of $R$. Suppose $S_i = C_R(r_i), 1 \leq i \leq t$ and $|S_1| \geq |S_2| \geq \dots \geq |S_t|$. Using Theorem \ref{thm2}, we have if $t = 3$ then $|\Cent(R)| = 4$ and if $t = 4$ then $|\Cent(R)| = 5$. So in both cases we get contradictions. Hence using Proposition \ref{thm1} and Proposition \ref{thm3}, we have $t = 5$ and for every non-central elements $r$ and $s$ of $R$, $C_R(r) = C_R(s)$ or $C_R(r) \cap C_R(s) = Z(R)$. Therefore, we have that
\[
|R| = \overset{5}{\underset{i = 1}{\sum}}|S_i| - 4|Z(R)|.
\]
Suppose $|S_1| < \frac{|R|}{4}$ that is $|S_1| \leq \frac{|R|}{5}$ as $1 < |S_1| \leq \frac{|R|}{2}$. Then we have $|R| \leq |R| - 4|Z(R)|$ which gives $|R| < |R|$, a contradiction.
Hence $|S_1| = \frac{|R|}{2}, \frac{|R|}{3}$ or $\frac{|R|}{4}$.

\noindent \textbf{Case $1$:} $|S_1| = \frac{|R|}{2}$. So $|R| \geq |S_1 + S_2| = \frac{(|R|/2)|S_2|}{|Z(R)|}$ which gives $|S_2| \leq 2|Z(R)|$ and hence $|S_2| = 2|Z(R)|$. So $|S_3| = |S_4| = |S_5| = 2|Z(R)|$. Therefore $|R| = \frac{|R|}{2} + 8|Z(R)| - 4|Z(R)|$ which gives $|R : Z(R)| = 8$.

\noindent \textbf{Case $2$:} $|S_1| = \frac{|R|}{3}$. So $|R| \geq |S_1 + S_2| = \frac{(|R|/3)|S_2|}{|Z(R)|}$ which gives $|S_2| \leq 3|Z(R)|$.

\textit{Subcase $(i)$:} $|S_2| = 2|Z(R)|$. Therefore $|S_3| = |S_4| = |S_5| = 2|Z(R)|$ and so $|R| = \frac{|R|}{3} + 2|Z(R)| + |S_3| + |S_4| + |S_5| - 4|Z(R)|$ which gives $|R : Z(R)| = 6$, a contradiction.

\textit{Subcase $(ii)$:} $|S_2| = 3|Z(R)|$. Therefore $|R| = \frac{|R|}{3} + 3|Z(R)| + |S_3| + |S_4| + |S_5| - 4|Z(R)|$ which gives $\frac{2|R|}{3} < |S_3| + |S_4| + |S_5| \leq 9|Z(R)|$. Therefore $|R| \leq 13|Z(R)|$ and hence using Theorem \ref{4cent}, $|R : Z(R)| = 8, 12$.

\noindent \textbf{Case $3$:} $|S_1| = \frac{|R|}{4}$. So $|R| \geq |S_1 + S_2| = \frac{(|R|/4)|S_2|}{|Z(R)|}$ which gives $|S_2| \leq 4|Z(R)|$.

\textit{Subcase $(i)$:} $|S_2| = 2|Z(R)|$. Therefore $|R| = \frac{|R|}{4} + 2|Z(R)| + |S_3| + |S_4| + |S_5| - 4|Z(R)|$ which gives $\frac{3|R|}{4} < |S_3| + |S_4| + |S_5| = 6|Z(R)|$. Therefore $|R| < 8|Z(R)|$, a contradiction.

\textit{Subcase $(ii)$:} $|S_2| = 3|Z(R)|$. Therefore $|R| = \frac{|R|}{4} + 3|Z(R)| + |S_3| + |S_4| + |S_5| - 4|Z(R)|$ which gives $\frac{3|R|}{4} < |S_3| + |S_4| + |S_5| \leq 9|Z(R)|$. Therefore $|R| < 12|Z(R)|$ and hence $|R : Z(R)| = 8$.

\textit{Subcase $(iii)$:} $|S_2| = 4|Z(R)|$. Therefore $|R| = \frac{|R|}{4} + 4|Z(R)| + |S_3| + |S_4| + |S_5| - 4|Z(R)|$ which gives $\frac{3|R|}{4} \leq 12|Z(R)|$. Therefore $|R| \leq 16|Z(R)|$ and hence $|R : Z(R)| = 8, 12, 16$.
 Hence the theorem follows.
\end{proof}

We conclude this section by the following theorem which  shows  that the converse of Theorem \ref{6cent} is not true.  
\begin{thm}
Let $R$ be a non-commutative ring. If $\frac{R}{Z(R)} \cong \Z_2 \times \Z_2 \times \Z_2$ then $|\Cent(R)| = 6$ or $8$. 
\end{thm}

\begin{proof}
Given $\frac{R}{Z(R)} \cong \Z_2 \times \Z_2 \times \Z_2$. Suppose $\frac{R}{Z(R)} = \{Z, r_1 + Z, \dots, r_7 + Z\}$, where $Z := Z(R)$. Then $|\Cent(R)| \leq 8$.

Suppose $\Cent(R) < 8$ then $C_R(r_i) = C_R(r_j)$ for some $1 \leq i, j \leq 7 \;(i \neq j)$. Now we consider $\frac{C_R(r_i)}{Z}$. Then possible orders of $\frac{C_R(r_i)}{Z}$ are $2$ and $4$. If $|\frac{C_R(r_i)}{Z}| = 2$ then $C_R(r_i) = Z \cup (r_i + Z)$ which is a contradiction, as $r_j, (r_i + r_j) \in C_R(r_i)$. So $|\frac{C_R(r_i)}{Z}| = 4$ and $\frac{C_R(r_i)}{Z} = \{Z, r_i + Z, r_j + Z, (r_i + r_j) + Z\} = \frac{C_R(r_j)}{Z}$. Next, we consider $\frac{C_R(r_i + r_j)}{Z}$. Here also, proceeding  in a similar way, we get $|\frac{C_R(r_i + r_j)}{Z}| = 4$ and $\frac{C_R(r_i + r_j)}{Z} = \{Z, r_i + Z, r_j + Z, (r_i + r_j) + Z\}$, as $r_i, r_j \in C_R(r_i + r_j)$. Therefore $C_R(r_i) = C_R(r_j) = C_R(r_i + r_j) = Z \cup (r_i + Z) \cup (r_j + Z) \cup (r_i + r_j + Z)$. Hence $|\Cent(R)| \leq 6$. If $|\Cent(R)| = 4$ then by Theorem \ref{4cent} (a), we have $|R : Z| = 4$, a contradiction. If $|\Cent(R)| = 5$ then by Theorem \ref{4cent} (b), we have $|R : Z| = 9$, a contradiction. Thus $|\Cent(R)| = 6$. 
\end{proof}

The group theoretic analogue of the above result can be found in \cite{ashrafi2006}.

\section{$7$-centralizer finite rings}


In this section, we give a characterization of $7$-centralizer finite rings. We begin with the following result.
\begin{lemma}\label{Lemma 4.1}
Let $R$ be a finite ring and $X = \{r_1, r_2, \dots, r_t\}$ be a set of pairwise non-commuting elements of $R$ having maximal size. If $|\Cent(R)| = t + 2$ then there exists a proper non-commutative centralizer $C_R(r)$ which contains $C_R(r_{i_1}), C_R(r_{i_2})$ and $C_R(r_{i_3})$ for three distinct $r_{i_1}, r_{i_2}, r_{i_3} \in \{1, 2, \dots, t\}$. 
\end{lemma}
\begin{proof}
Suppose, there exists an $r_i \in X$ such that $S := C_R(r_i)$ is non-commutative. So $|\Cent(S)| \geq 4$ and $S$ contains at least three proper centralizers, say $C_S(s_j), j = 1, 2, 3$. Therefore by hypothesis $C_R(r_i) \neq C_R(s_j)$ for every $i \in \{1, 2, \dots, t\}$ and $j \in \{1, 2, 3\}$. 
Hence $|\Cent(R)| > t + 2$, a contradiction. Thus each $C_R(r_i), 1 \leq i \leq t$ is commutative.

Next, suppose for an element $r \in R$, $C_R(r)$ is commutative. So, there exists an index $j \in \{1, 2, \dots, t\}$ such that $r \in C_R(r_j)$. Now for any element $s \in C_R(r)$ we have $s \in C_R(r_j)$, as $C_R(r_j)$ is commutative. Therefore $C_R(r) \subseteq C_R(r_j)$. Similarly, it can be seen that $C_R(r_j) \subseteq C_R(r)$. Hence for each element $r \in R$, $C_r(R)$ is commutative if and only if $C_R(r) = C_R(r_j)$ for some $i \in \{1, 2, \dots, t\}$. Therefore there exists a proper centralizer $C_R(r) := T$ such that $T$ is not commutative, as $|\Cent(R)| = t + 2$. So $|\Cent(T)| \geq 4$. Suppose $C_T(r_{i_1}), C_T(r_{i_2})$ and $C_T(r_{i_3})$ are three proper centralizers of $T$, so these are three proper centralizers of $R$. Hence $C_R(r)$ contains $C_R(r_{i_1}), C_R(r_{i_2})$ and $C_R(r_{i_3})$ for three distinct $r_{i_1}, r_{i_2}, r_{i_3} \in \{1, 2, \dots, t\}$, as $|\Cent(R)| = t + 2$. Thus the lemma follows.  
\end{proof}

The next result, which is analogous to \cite[Lemma $2.7$]{ASA2007}, is also useful. 
\begin{lemma}\label{powerof2}
Let $R$ be a $7$-centralizer ring. Then $|R : Z(R)|$ is not a power of $2$.
\end{lemma}
\begin{proof}
Suppose that $|R : Z(R)|$ is a power of $2$. Let $ \{r_1, r_2, \dots, r_t\}$ be a set of pairwise non-commuting 
elements of $R$ having maximal size such that $|R : C_R(r_i)| = \beta _i, 1 \leq i \leq t$ with $\beta_1 \leq \beta_2 \leq \dots \leq \beta_t$. By Theorem \ref{thm2}, if $t = 3$ then $|\Cent(R)| = 4$ and if $t = 4$ then $|\Cent(R)| = 5$. So in both cases we get contradictions. Hence, by Proposition \ref{thm1}, we have $t = 5$ or $6$.

Suppose $t = 5$. Then by Theorem \ref{tomtheorem}, we have 
$\beta_2 \leq 4$. Also by Proposition $2.5$ of \cite{ASA2007}, there exists a proper non-commutative centralizer $C_R(r)$ which contains $C_R(r_{i_1}), C_R(r_{i_2})$ and $C_R(r_{i_3})$ for three distinct $i_1, i_2, i_3 \in \{1, 2, \dots, 5\}$. So $C_R(r_i) \cap C_R(r_j) = Z(R)$ for all $i, j \in \{1, 2, \dots, 5\}$, otherwise $\exists \;l \in C_R(r_i) \cap C_R(r_j) - Z(R)$ such that $C_R(l) \neq C_R(r_i)$ for all $i$ and $|\Cent(R)|$ will be atleast $5 + 1 + 1 + 1 = 8$, a contradiction. If $\beta_2 = 3$ then 
\[
|R| \geq |C_R(r_1) + C_R(r_2)| = \frac{|C_R(r_1)||C_R(r_2)|}{|C_R(r_1) \cap C_R(r_2)|} \geq \frac{(|R|/3)(|R|/3)}{|Z(R)|}
\] 
which gives $|R : Z(R)| \leq 9$ and so $|R : Z(R)| = 9$, as $3$ divides $|R : Z(R)|$. Therefore, by Theorem \ref{4cent}(b), we have $|\Cent(R)| = 5$, a contradiction. If $\beta_2 = 2$ then $\beta_1 = \beta_2 = 2$. Therefore 
\[
|R| \geq |C_R(r_1) + C_R(r_2)| = \frac{|C_R(r_1)||C_R(r_2)|}{|C_R(r_1) \cap C_R(r_2)|} = \frac{(|R|/2)(|R|/2)}{|Z(R)|}
\]
 which gives $|R : Z(R)| = 4$. Thus, by Theorem \ref{4cent} (a), we have $|\Cent(R)| = 4$, a contradiction. Hence $ \beta_2 = 4$. So by Lemma \ref{tomlemma}, we have $\beta_i = 4$ for $i \geq 2$. Therefore by Lemma \ref{Lemma 4.1} there exists a proper non-commutative centralizer $C_R(r)$ which contains at least three $C_R(r_i)$'s, say $i = 2, 3, 4$. Therefore $R = C_R(r_1) \cup C_R(r_5) \cup C_R(r)$ and so by Theorem \ref{tomtheorem}, 
$\beta_5 = 2$, a contradiction.

Now suppose that $t = 6$. Therefore, by Proposition \ref{thm3}, we have every proper centralizer of $R$ is commutative and for every non-central elements $r$ and $s$ of $R$, $C_R(r) = C_R(s)$ or $C_R(r) \cap C_R(s) = Z(R)$.
 Now for any $r + Z(R) \in \frac{C_R(r_i)}{Z(R)} \cap \frac{C_R(r_j)}{Z(R)}$ we have $r \in Z(R)$. So for any $r_i, r_j \in \{1, 2, \dots, 6\}, i \neq j$, we have $\frac{C_R(r_i)}{Z(R)} \cap \frac{C_R(r_j)}{Z(R)} = \{Z(R)\}$. We assume that $|\frac{C_R(r_i)}{Z(R)}| = k_i$ for $1 \leq i \leq 6$. Also for any $r \in R$ we have $r \in C_R(r_i)$, for some $i$. Therefore $\frac{R}{Z(R)} = \overset{6}{\underset{i = 1}{\cup}} \frac{C_R(r_i)}{Z(R)}$ this gives $|\frac{R}{Z(R)}| = \overset{6}{\underset{i = 1}{\sum}} |\frac{C_R(r_i)}{Z(R)}| - 5$ so $\overset{6}{\underset{i = 1}{\sum}} k_i = |\frac{R}{Z(R)}| + 5$. Hence $\overset{6}{\underset{i = 1}{\sum}} k_i$ is an odd integer, which is a contradiction, as all $k_i$s are even. Thus the lemma follows.
\end{proof}

Now we give the main result of this section.
\begin{thm}\label{7cent}
Let $R$ be a $7$-centralizer ring. Then $|R : Z(R)| = 12, 18, 20, 24$ or $25$.
\end{thm}  
\begin{proof}
Given $|\Cent(R)| = 7$. Let $\{r_1, r_2, \dots, r_t\}$ be a set of pairwise non-commuting elements of $R$ having maximal size. Then $C_R(r_i), 1 \leq i \leq t$ are distinct proper centralizers of $R$. Suppose $S_i = C_R(r_i), 1 \leq i \leq t$ and $|S_1| \geq |S_2| \geq \dots \geq |S_t|$. Using Theorem \ref{thm2}, if $t = 3$ then $|\Cent(R)| = 4$ and if $t = 4$ then $|\Cent(R)| = 5$. So in both cases we get contradictions. Hence using Proposition \ref{thm1}, we have $t = 5$ or $6$.

Suppose $t = 5$. Then by Lemma \ref{Lemma 4.1} there exists a proper non-commutative centralizer $C_R(r)$ which contains $C_R(r_{i_1}), C_R(r_{i_2})$ and $C_R(r_{i_3})$ for three distinct $i_1, i_2, i_3 \in \{1, 2, \dots, 5\}$. So $C_R(r_i) \cap C_R(r_j) = Z(R)$ for all $i, j \in \{1, 2, \dots, 5\}$, otherwise $\exists \;l \in C_R(r_i) \cap C_R(r_j) - Z(R)$ such that $C_R(l) \neq C_R(r_i)$ for all $i$ and $|\Cent(R)|$ will be at least $5 + 1 + 1 + 1 = 8$, a contradiction. Now using Theorem \ref{tomtheorem}, if $|R : C_R(r_i)| = \beta_i, 1 \leq i \leq 5$ and $\beta_1 \leq \beta_2 \leq \dots \leq \beta_5$, then $\beta_2 \leq 4$. Suppose $\beta_2 \leq 3$ then $|R| \geq |C_R(r_1) + C_R(r_2)| = \frac{|C_R(r_1)||C_R(r_2)|}{|C_R(r_1) \cap C_R(r_2)|} \geq \frac{(|R|/3)(|R|/3)}{|Z(R)|}$ gives $|R : Z(R)| \leq 9$ and so by Lemma \ref{powerof2}, we have $|R : Z(R)| = 4, 9$. Therefore, by Theorem \ref{4cent}, we have $|\Cent(R)| = 4$ or $5$, a contradiction. If $\beta_2 \leq 4$, in a similar way it can be seen that $|R : Z(R)| \leq 16$. So using Lemma \ref{powerof2}, Theorem \ref{4cent}, we have $|R : Z(R)| = 12$.

Suppose $t = 6$ then using Proposition \ref{thm3}, we have for every non-central elements $r$ and $s$ of $R$, $C_R(r) = C_R(s)$ or $C_R(r) \cap C_R(s) = Z(R)$. Therefore we have
\[
|R| = \overset{6}{\underset{i = 1}{\sum}}|C_R(i)| - 5|Z(R)|.
\]
Suppose $|S_1| < \frac{|R|}{5}$ that is $|S_1| \leq \frac{|R|}{6}$ as $1 < |S_1| \leq \frac{|R|}{2}$. Then we have $|R| \leq |R| - 5|Z(R)|$ which gives $|R| < |R|$, a contradiction.
Hence $|S_1| = \frac{|R|}{2}, \frac{|R|}{3}, \frac{|R|}{4}$ or $\frac{|R|}{5}$.

\noindent \textbf{Case $1$:} $|S_1| = \frac{|R|}{2}$. So $|R| \geq |S_1 + S_2| = \frac{(|R|/2)|S_2|}{|Z(R)|}$ which gives $|S_2| \leq 2|Z(R)|$ and hence $|S_2| = 2|Z(R)|$. Similarly it can be seen that $|S_3| = |S_4| = |S_5| = |S_6| = 2|Z(R)|$. Therefore $|R| = \frac{|R|}{2} + 10|Z(R)| - 5|Z(R)|$ which gives $|R : Z(R)| = 10$ and hence $R$ is commutative, a contradiction.

\noindent \textbf{Case $2$:} $|S_1| = \frac{|R|}{3}$. So $|R| \geq |S_1 + S_2| = \frac{(|R|/3)|S_2|}{|Z(R)|}$ which gives $|S_2| \leq 3|Z(R)|$.

\textit{Subcase $(i)$:} $|S_2| = 2|Z(R)|$. Therefore $|S_3| = |S_4| = |S_5| = |S_6| = 2|Z(R)|$ which gives $|R| =  \frac{|R|}{3} + 10|Z(R)| - 5|Z(R)|$ and so $|R|$ is not divisible by $|Z(R)|$, a contradiction.

\textit{Subcase $(ii)$:} $|S_2| = 3|Z(R)|$. Therefore $|R| = \frac{|R|}{3} + 3|Z(R)| + |S_3| + |S_4| + |S_5| + |S_6| - 5|Z(R)|$ which gives $\frac{2|R|}{3} < |S_3| + |S_4| + |S_5| + |S_6|\leq 12|Z(R)|$. Therefore $|R| < 18|Z(R)|$ and hence using Lemma \ref{powerof2}, Theorem \ref{4cent}, we have $|R : Z(R)| = 12$.

\noindent \textbf{Case $3$:} $|S_1| = \frac{|R|}{4}$. So $|R| \geq |S_1 + S_2| = \frac{(|R|/4)|S_2|}{|Z(R)|}$ which gives $|S_2| \leq 4|Z(R)|$.

\textit{Subcase $(i)$:} $|S_2| = 2|Z(R)|$. Therefore $|R| = \frac{|R|}{4} + 2|Z(R)| + |S_3| + |S_4| + |S_5| + |S_6| - 5|Z(R)|$ which gives $\frac{3|R|}{4} < |S_3| + |S_4| + |S_5| + |S_6| = 8|Z(R)|$. Therefore $|R| \leq 10|Z(R)|$ and hence $|R : Z(R)| = 4, 8, 9$. So using Lemma \ref{powerof2}, Theorem \ref{4cent}, we have $|\Cent(R)| = 4$ or $5$, a contradiction.

\textit{Subcase $(ii)$:} $|S_2| = 3|Z(R)|$. Therefore $|R| = \frac{|R|}{4} + 3|Z(R)| + |S_3| + |S_4| + |S_5| + |S_6| - 5|Z(R)|$ which gives $\frac{3|R|}{4} < |S_3| + |S_4| + |S_5| + |S_6| \leq 12|Z(R)|$. Therefore $|R| \leq 15|Z(R)|$ and hence using Lemma \ref{powerof2}, Theorem \ref{4cent}, we have $|R : Z(R)| = 12$.

\textit{Subcase $(iii)$:} $|S_2| = 4|Z(R)|$. Therefore $|R| = \frac{|R|}{4} + 4|Z(R)| + |S_3| + |S_4| + |S_5| +  |S_6| - 5|Z(R)|$ which gives $\frac{3|R|}{4} < 16|Z(R)|$. Therefore $|R| \leq 21|Z(R)|$ and hence using Lemma \ref{powerof2}, Theorem \ref{4cent}, we have $|R : Z(R)| =  12, 18, 20$.
 
\noindent \textbf{Case $4$:} $|S_1| = \frac{|R|}{5}$. So $|R| \geq |S_1 + S_2| = \frac{(|R|/5)|S_2|}{|Z(R)|}$ which gives $|S_2| \leq 5|Z(R)|$.

\textit{Subcase $(i)$:} $|S_2| = 2|Z(R)|$. Therefore $|R| = \frac{|R|}{5} + 2|Z(R)| + |S_3| + |S_4| + |S_5| + |S_6| - 5|Z(R)|$ which gives $\frac{4|R|}{5} < 8|Z(R)|$. Therefore $|R| \leq 9|Z(R)|$, a contradiction.

\textit{Subcase $(ii)$:} $|S_2| = 3|Z(R)|$. Therefore $|R| = \frac{|R|}{5} + 3|Z(R)| + |S_3| + |S_4| + |S_5| + |S_6| - 5|Z(R)|$ which gives $\frac{4|R|}{5} < |S_3| + |S_4| + |S_5| + |S_6| \leq 12|Z(R)|$. Therefore $|R| < 15|Z(R)|$ and hence using Lemma \ref{powerof2}, Theorem \ref{4cent}, we have $|R : Z(R)| = 12$.

\textit{Subcase $(iii)$:} $|S_2| = 4|Z(R)|$. Therefore $|R| = \frac{|R|}{5} + 4|Z(R)| + |S_3| + |S_4| + |S_5| +  |S_6| - 5|Z(R)|$ which gives $\frac{4|R|}{5} < 16|Z(R)|$. Therefore $|R| < 20|Z(R)|$ and hence using Lemma \ref{powerof2}, Theorem \ref{4cent}, we have $|R : Z(R)| = 12, 18$.

\textit{Subcase $(iv)$:} $|S_2| = 5|Z(R)|$. Therefore $|R| = \frac{|R|}{5} + 5|Z(R)| + |S_3| + |S_4| + |S_5| +  |S_6| - 5|Z(R)|$ which gives $\frac{4|R|}{5} \leq 20|Z(R)|$. Therefore $|R| \leq 25|Z(R)|$ and hence using Lemma \ref{powerof2}, Theorem \ref{4cent}, we have $|R : Z(R)| = 12, 18, 20, 24, 25$.
 This completes the proof.
   \end{proof}

We conclude this paper by noting that the converse of the above theorem is true if $|R: Z(R)| = 25$.




\begin{thebibliography}{3}





\bibitem{ASA2007}
A. Abdollahi, S. M. J. Amiri and A. M. Hassanabadi, Groups with specific number of centralizers, {\em Houston Journal of Mathematics}, {\bf 33} (1) (2007), 43--57.




\bibitem{ashrafi00}
A. R. Ashrafi,  On finite groups with a given number of centralizers, {\em Algebra Colloq.}, {\bf 7} (2) (2000), 139--146.

\bibitem{ashrafi2000}
A. R. Ashrafi, Counting the centralizers of some finite groups, {\em Korean J. Comput. \& Appl. Math}, {\bf 7} (1) (2000), 115--124.

\bibitem{ashrafi2006}
A. R. Ashrafi and B. Taeri, On finite groups with exactly seven element centralizers, {\em K J. Appl. Math \& Computing}, {\bf 22} (1-2) (2006), 403--410.

\bibitem{baishya}
S. J. Baishya, On finite groups with specific number of
centralizers, {\em Int. Electron. J. Algebra}, {\bf 13}, 53--62, 2013. 






\bibitem{bBhK14}
M. Behboodi, R. Beyranvand, A. Hashemi and H. Khabazian, Classification of finite rings: theory and algorithm, {\em Czechoslovak Math. J.}, {\bf 64}(3) (2014), 641--658. 

\bibitem{bG94}
S. M. Belcastro and G. J. Sherman,  Counting centralizers in finite groups, {\em Math. Magazine}, {\bf 67} (5) (1994), 366--374.






\bibitem{Cu10}
C. J. Chikunji, A Classification of a certain class of completely primary finite rings, Ring and Module Theory, Trends in Mathematics 2010, {\em Springer Basel AG}, 83--90.

\bibitem{dOP94}
J. B. Derr, G. F. Orr and P. S. Peck, Noncommutative rings of order $p^4$, {\em J.  Pure Appl. Algebra}, {\bf 97}(2) (1994), 109--116.


\bibitem{Dutta13}
J. Dutta,  A characterization of $4$-centralizer groups, {\em Chinese Journal of Mathematics}, {\bf 2013} (2013), Article ID 871072, 2 pages.

\bibitem{Dutta14}
J. Dutta, A characterization of $4$-centralizer rings, Preprint.


\bibitem{Jdutta}
J. Dutta, D. K. Basnet and R. K. Nath, Characterizing some rings of finite order, Preprint.




\bibitem{El68}
K. E. Eldridge, Orders for finite noncommutative rings with unity, {\em Amer. Math. Monthly}, {\bf 75} (1968), 512--514.

\bibitem{Fin93}
B. Fine, Classification of finite rings of order $p^2$, {\em Math. Magazine}, {\bf 66}(4)  (1993), 248--252.

\bibitem{gC95}
R.W. Goldbach and H.L. Claasen, Classification of not commutative rings with identity of order dividing $p^4$, {\em Indag. Math.}, {\bf 6} (1995), 167--187.

\bibitem{dmachale}
D. MacHale,     Commutativity in finite rings,  {\em Amer. Math. Monthly}, 83(1976), 30--32.




\bibitem{Ra70}
R. Raghavendran,  A class of finite rings, {\em Compositio Math.}, {\bf 22} (1970), 49--57.

\bibitem{tom87}
M. J. Tomkinson,  Groups covered by finitely many cosets or subgroups, {\em Comm. Algebra}, {\bf 15} (4) (1987), 845--859.


\end{thebibliography}
\end{document}